\documentclass[pdflatex,sn-mathphys]{sn-jnl}% Math and Physical Sciences Reference Style
\usepackage{amscd}

\usepackage{rotating}
\usepackage{verbatim}
\usepackage{mathtools}
\usepackage{thm-restate}
\usepackage[shortlabels]{enumitem}
\usepackage{stmaryrd}

\usepackage[utf8]{inputenc}

\usepackage{thmtools}
\usepackage{hyperref}
\usepackage{cleveref}

\usepackage{color}
\renewcommand{\x}{\times}

\newcommand{\N}{\mathbb{N}}

\newcommand{\Cant}{2^\N}
\newcommand{\Cyl}[1]{\ensuremath{\llbracket #1 \rrbracket}}
\newcommand{\uhr}{\upharpoonright}
\newcommand{\Conc}{\ensuremath{\,\mbox{}^\frown\,}}
\jyear{2021}%

\theoremstyle{thmstyleone}%
\newtheorem{Thm}[subsection]{Theorem}
\newtheorem{Fact}[subsection]{Fact}
\newtheorem{Lemma}[subsection]{Lemma}
\newtheorem{Corollary}[subsection]{Corollary}
\newtheorem{Proposition}[subsection]{Proposition}

\theoremstyle{thmstyletwo}%

\theoremstyle{thmstylethree}%
\newtheorem{Def}[subsection]{Definition}
\newtheorem*{Notation*}{Notation}
\newtheorem{Construction}[subsection]{Construction}
\begin{document}

\title[Turing Degrees and Randomness for Continuous Measures]{Turing Degrees and Randomness for Continuous Measures}


\author[]{\fnm{Mingyang} \sur{Li}}\email{mingyang.psu@gmail.com}

\author[1]{\fnm{Jan} \sur{Reimann}}\email{jan.reimann@psu.edu}
%%\equalcont{These authors contributed equally to this work.}

\affil[1]{\orgdiv{Department of Mathematics}, \orgname{Penn State University}, \orgaddress{\city{University Park}, \postcode{16802
    }, \state{PA}, \country{USA}}}


\abstract{We study degree-theoretic properties of reals that are not random with respect to any continuous probability measure (NCR). To this end, we introduce a family of generalized Hausdorff measures based on the iterates of the ``dissipation'' function of a continuous measure and study the effective nullsets given by the corresponding Solovay tests. We introduce two constructions that preserve non-randomness with respect to a given continuous measure. This enables us to prove the existence of NCR reals in a number of Turing degrees. In particular, we show that every $\Delta^0_2$-degree contains an NCR element.}




\maketitle

\section{Introduction}

Martin-Löf's 1966 paper~\cite{martin1966definition} put the notion of an
individual random sequence on a sound mathematical footing. He gave a rigorous
definition of what it means for an infinite binary sequence (which we also refer
to as a \emph{real}) to be random with respect to a Bernoulli measure. Zvonkin
and Levin~\cite{zvonkin1970complexity} extended the definition to computable
measures on $\Cant$ and showed that every non-computable real $X \in \Cant$ that
is random with respect to computable probability measure is Turing equivalent to
a sequence random with respect to Lebesgue measure on $\Cant$, the measure
induced by a fair coin toss on $\{0,1\}$. This marked one of the first results
connecting randomness and the degrees of unsolvability. Over the following
decades, our understanding of how randomness (in the sense of Martin-Löf and
related, algorithmically based notions) and computability interact has grown
tremendously. Two recent monographs attest to
this~\cite{Nies:2009a,downey2010algorithmic}. However, most investigations
focused on the computational properties sequences that \emph{are} random with
respect to some kind of measure: Lebesgue measure (the vast majority of
results), but also other computable probability measures and Hausdorff measures.
This leaves the question whether we can characterize, in terms of computability
theory, the reals which do not exhibit any random behavior \emph{at all}. The
notion of ``being far from random'' so far has mostly been studied from the
point of view of \emph{triviality} and \emph{lowness}, which characterize reals
by having low initial-segment Kolmogorov complexity or by having little
derandomization power as oracles, respectively. We again refer to the
monographs~\cite{Nies:2009a,downey2010algorithmic} for an overview of a large
number of results in this direction.

This paper focuses on a different kind of question: {\em Given a real $X \in
  \Cant$, and a family of probability measures $\mathcal{M}$, is $X$ random with
respect to a measure in $\mathcal{M}$, and if not, what is the computational
power of $X$?}

Levin~\cite{levin1976uniform} was the first to define Martin-Löf randomness for
arbitrary probability measures. Levin defined \emph{uniform tests} of
randomness. Such a test is a left-enumerable function $t$ that maps pairs of
measures and reals to non-negative real numbers or infinity such that for any
probability measure $\mu$ on $\Cant$, $\int t(\mu,X) d\mu(X) \leq 1$. A sequence
$X$ is random for $\mu$ if for all uniform test $t$, $t(\mu,X) < \infty$. A
different approach to randomness with respect to arbitrary measures was given by
Reimann and Slaman~\cite{reimann2015measures}. Their approach represents
measures as reals and makes these available as oracles in relativized Martin-Löf
tests. We will present more details on this approach in Section~\ref{sec:2}. Day
and Miller~\cite{day2013randomness} showed that the two approaches are
equivalent, that is, they define the same set of random reals.

It is a trivial fact that any real $X$ that is an \emph{atom} of a measure
$\mu$, i.e., $\mu\{X\} >0$, is random with respect to $\mu$. Reimann and
Slaman~\cite{reimann2015measures} showed that a real $X$ is non-trivially random
with respect to some probability measure $\mu$ if and only if $X$ is
non-computable. In other words, if we do not further restrict the family of
probability measures, a real has \emph{some} non-trivial random content if and
only if it is not computable. Day and Miller~\cite{day2013randomness} gave an
alternative prof of this result using Levin's neutral measures (a single measure
relative to which \emph{every} sequence is random).

A more intricate structure emerges when we ask which sequences are random with
respect to a \emph{continuous}, i.e.\ non-atomic, probability measure. Reimann
and Slaman~\cite{reimann2015measures} showed that if a sequence $X \in \Cant$ is
not $\Delta^1_1$, it is random with respect to a continuous measure. We use the
term \emph{NCR} to denote those reals which are not random with respect to any
continuous measure. Kjos-Hanssen and Montalb\'{a}n~\cite{Montalban:2005a} showed
every member of a countable $\Pi^0_1$ set of sequence is NCR. Cenzer, Clote,
Smith, Soare, and Wainer~\cite{cenzer1986members} showed that members of
countable $\Pi^0_1$ sets of reals exist in every Turing degree
$\mathbf{0}^{(\alpha)}$, where $\alpha$ is any computable ordinal. Therefore,
the Kjos-Hanssen-Montalb\'{a}n result implies the set of NCR reals is cofinal in
$\Delta^1_1$ Turing-degrees.

On the other hand, Barmpalias, Greenberg, Montalb\'{a}n and
Slaman~\cite{barmpalias2012k} connected computational lowness with NCR by
showing that any real Turing below an incomplete r.e.\ degree is NCR. In
particular, every K-trivial is NCR. Their result makes use of a close
connection between the granularity function of a continuous measure (introduced
in the next section) and the settling time of a $\Delta^0_2$ real, which was
first observed by Reimann and Slaman~\cite{reimann-slaman:unpub}. The
granularity function (along with its ``companion'', the dissipation function of
a meaure), will also play a central role in this paper.

The previous results suggest an attempt to classify the $\Delta^1_1$ Turing
degress along the following lines:
\begin{enumerate}[(1)]
  \item Which Turing degrees consist entirely of NCR reals?
  \item Which Turing degrees do not contain any NCR reals?
  \item Which Turing degrees contain NCR reals?
\end{enumerate}

Haken~\cite{haken2014randomizing} studied these questions with respect to
stronger randomness notions for arbitrary (not necessarily continuous) measures,
in particular difference and weak-$n$-randomness for $n \geq 2$. He also linked
continuous randomness to higher randomness by showing that NCR reals are not
3-randomizable, i.e. for any (possibly atomic) measure $\mu$ and any
representation $R_\mu$ of $\mu$, NCR reals are not $\mu$-random with respect to
any Martin-L\"{o}f $\mu$-test relative to $R_\mu''$.

Regarding Question (2), Reimann and Slaman~\cite{reimann2018effective} showed
that every real Turing below a (Lebesgue) $3$-random real and not recursive in
$0'$ is random with respect to a continuous measure.

In this paper, we mainly focus on Question (3). We construct NCR reals in
certain families of Turing degrees. Our main technique is to recursively
approximate non-random reals using other non-random reals which are, in a
certain sense, even ``more non-random''. For this purpose, we quantify
non-randomness with respect to a given measure. We introduce a new randomness
test parameterized by a natural number $n$ which corresponds to the level of
non-randomness. We should point out that the level $n$ of non-randomness we
define in this paper is not related to the notion of Martin-L\"{o}f
$n$-randomness.

This paper is organized as follows. In Section~\ref{sec:2}, we introduce the new randomness test which quantifies the level of non-randomness and prove some basic facts about it which we will need later. In Sections~\ref{sec:3} and~\ref{sec:4}, respectively, we present two constructions of reals based on levels of non-randomness, one for reals recursively enumerable in and above (r.e.a.\ ) a given real, the other one for reals with a self-modulus.  Finally, in Section~\ref{sec:5}, we infer the existence of NCR reals in certain Turing degrees using the construction in Sections~\ref{sec:3} and~\ref{sec:4}. In particular, our constructions can be used to prove the following theorem.

\bigskip
% \begin{restatable}[Main result]{Thm}{mainone}~
\begin{restatable}{Thm}{mainone}~
  \label{Thm:main1}
  \begin{enumerate}[(a)]
    \item Any n-REA Turing degree contains an NCR real.
    \item Any self-modulus degree contains an NCR real.
  \end{enumerate}
\end{restatable}

The theorem in particular implies
\begin{restatable}{Corollary}{maintwo}
  \label{Thm:main2}
  Every $\Delta^0_2$ degree contains an NCR real.
\end{restatable}

\subsection*{Acknowledgments}
We would like to thank Ted Slaman for many helpful discussions, and for first
observing the relation between the granularity function of a measure and the
settling time of a real. This crucial initial insight inspired much of the work
presented here.

\subsection*{Notation}
In the following, we list the notation used in this paper. The reader can refer to \cite{soare2016turing} for more details.

\begin{itemize}
  \item We use $\log$ to denote the binary logarithm.
  \item Lower case Roman letters denote natural numbers,
        except $f,g,h$ (and sometimes $s,t$), which denote functions.
  \item If $f$ is a function and $n \geq 1$, $f^{(n)}$ denotes its \emph{$n$-th iterate}, i.e.\ $f^{(1)} = f$ and $f^{(n+1)} = f \circ f^{(n)}$.
  \item We use capital Roman letters $X,Y,Z,A,B,C,R$ to denote set of natural numbers as well as infinite binary strings (reals).
  \item We use lowercase Greek letters $\sigma$,$\tau$ to denote finite binary strings. The length of a string $\sigma$ will be denoted by $\mid\sigma\mid$. We use $\Cyl{\sigma}$ to denote the set of all infinite binary strings extending $\sigma$.
  \item We use $\operatorname{dom}(f)$ to denote the domain of a partial recursive function $f$.
  \item We fix an effective enumeration $\{\Phi_i\}$ of all oracle Turing machines.
  \item We use $\Phi^A_e$ to denote the machine with oracle $A$ and G{\"o}del number $e$. We write $\Phi^A_e(x) = y$ if the machine halts on input $x$ and outputs $y$. If $\Phi^A_e(x)$ does not halt, we write $\Phi^A_e(x) = \uparrow$. Finally, we let $W^A_e = \operatorname{dom}(\Phi^A_e)$.
  \item We use $\Phi^A_{e,k}(x)$ to denote the $e$-th machine with oracle $A$  running for $k$ steps. Without loss of generality, $\Phi^A_{e,k}(x)=\uparrow$ when $x>k$. We put $W^A_{e,s} =\operatorname{dom}(\Phi^A_{e,s})\uhr_s$.
  \item We use $\sigma \Conc \tau$ to denote the concatenation of strings $\sigma$ and $\tau$.
\end{itemize}

\section{Quantifying non-randomness}
\label{sec:2}

In this section, we first briefly review the definition of randomness with respect to arbitrary measures given by \cite{reimann2015measures}. We refer the readers for \cite{reimann2015measures} and \cite{day2013randomness} for more details.

First of all, we define a metrical structure on the set of all probability measure on $2^\omega$.
\begin{Def}
  For any probability measures $\mu$ and $\nu$ on $2^\omega$, define the \textit{distance function} $d(\mu,\nu)$ as
  $$d(\mu,\nu)=\sum_{\sigma\in 2^{<\omega}} 2^{-\mid\sigma\mid}\mid\mu\Cyl{\sigma}-\nu\Cyl{\sigma}\mid.$$
\end{Def}
Let ${\mathcal{P}}(2^\omega)$ be the set of all probability measures on
$2^\omega$, and let $\mu_\sigma$ be the measure which is identical with the
characteristic function of the principal filter of $\{\sigma\Conc 0^\omega\}$, that is, for any
$H\subset 2^\omega$,
\begin{equation*}
  \mu_\sigma(H) = \begin{cases}
    1 & \text{if $\sigma\Conc 0^\omega \in H$,}     \\
    0 & \text{if $\sigma\Conc 0^\omega \notin H$. }
  \end{cases}
\end{equation*}

The following properties hold.

\begin{Proposition}~
  \begin{enumerate}[(1)]
    \item $d(\mu,\nu)$ is a metric on ${\mathcal{P}}(2^\omega)$.
    \item ${\mathcal{P}}(2^\omega)$ with the topology generated by $d(\mu,\nu)$ is a Polish space.
    \item The closure of all $\mu_\sigma$ under binary average forms a countable dense subset of $({\mathcal{P}}(2^\omega),d)$.
  \end{enumerate}
\end{Proposition}

For the proof, refer to \cite{reimann2015measures} or \cite{day2013randomness}.
The proposition allows for  representing any element of ${\mathcal{P}}(2^\omega)$ by a Cauchy sequences of elements in (3). Let us assume $\{\mu_0,\mu_1,\mu_2,\ldots\}$ is a fixed effective enumeration of the set in (3). Any sequence of measures in (3) can then be represented by its sequence of indices in $\{\mu_0,\mu_1,\mu_2,\ldots\}$. If one develops this correspondence carefully it is possible to prove the following~\cite{day2013randomness}.

% now we only need to enumerate the subscripts of Cauchy sequences. To make the statement easier, we use use-monotone Turing functional defined as below rather than the usual Turing functional in this part.
% \begin{Def}
% Let $\Gamma$ be a computably enumerable set of triple $(m,k,\sigma)$ where $m$ and $k$ are natural numbers, and $\sigma$ is a finite binary sequence. We say $\Gamma$  is a \textit{use-monotone Turing functional} if it satisfy the following.

% (1) For any triples $(m_1,k_1,\sigma_1)$ and $(m_2,k_2,\sigma_2)$ in $\Gamma$, if $m_1=m_2$ and $\sigma_1$, $\sigma_2$ are compatible, then $k_1 = k_2$ and $\sigma_1=\sigma_2$.

% (2) For any triples $(m_1,k_1,\sigma_1)$ and $(m_2,k_2,\sigma_2)$ in $\Gamma$, if $\sigma_1$ is a proper initial segment of $\sigma_2$, then $m_1< m_2$.

% (3) For any triples $(m_1,k_1,\sigma_1)$ in $\Gamma$, then for any $m_2<m_1$, there is a triple $(m_2,k_2,\sigma_2)\in \Gamma$ such that $\sigma_2$ is a initial segment of $\sigma_1$.
% \end{Def}
% In the above definition, property (1) suggests that for fixed $m$ and $\sigma$, $k$ is unique if exists. Furthermore, for a fixed $m$ and real $X$, if there exists an initial segment $\sigma\sqsubset X$ such that \ $(m,k,\sigma)\in\Gamma$ for some natural number $k$, then the initial segment $\sigma$ is also unique. We write $\Gamma^X(m)=k$ or $\Gamma^X(m)\downarrow$ if such an initial segment $\sigma$ and natural number $k$ exist, and $\Gamma^X(m)\uparrow$ if not.

% In \cite{day2013randomness}, Day and Miller show the following property. 

\begin{Proposition}
  \label{Pro:repre}
  There exists a Turing functional $\Gamma$, such that for any real $X$ and any natural number $n$, $\Gamma^X(n)\downarrow$, and the following hold.
  \begin{enumerate}
    \item $d(\mu_{\Gamma^X(n)},\mu_{\Gamma^X(n+1)})\leq 2^{-n}$;
    \item the function $\rho: 2^\omega\rightarrow {\mathcal{P}}(2^\omega)$ defined as
          $$\rho(X)=\lim_n \mu_{\Gamma^X(n)}$$
          is a continuous surjection.

    \item for any $X$, $\rho^{-1}(\{\rho(X)\})$ is $\Pi^0_1(X)$.
  \end{enumerate}
\end{Proposition}

% With the above proposition, we can define the representation for any probability measure.
From now on, we fix a mapping $\rho$ as given by Proposition~\ref{Pro:repre}.

\begin{Def}
  A \textit{representation} of a probability measure $\mu$ is a real $R$ such that $\rho(R)=\mu$.
\end{Def}

Note that for a given probability measure $\mu$, its representation might not be unique. However, any representation of $\mu$ can compute a two-sided effective approximation to $\mu\Cyl{\sigma}$, for any given $\sigma$.

Using representations as oracles, one can define randomness tests and computability relative to a given probability measure.

\begin{Def}
  A \textit{Martin-L\"of-$\mu$-test relative to a representation $R_\mu$(or simply
    Martin-L\"of-$R_\mu$-test)} is a sequence of uniformly $\Sigma^0_1(R_\mu)$ sets
  $(V_n)_{n\in \N}$ such that for all $n$, $\mu(V_n)\leq2^{-n}$. \\
  $X\in 2^\omega$ \textit{passes} a Martin-L\"of-$R_\mu$-test if $X\notin\cap_{n\in \omega} V_n$.\\
  For any probability measure $\mu$ on $2^\omega$ and a representation $R_\mu$ of $\mu$, $X\in 2^\omega$ is \textit{$R_\mu$-$\mu$-random} if $X$ passes every Martin-L\"{o}f-$\mu$ test relative to $R_\mu$.
\end{Def}

\begin{Def}
  A set or function is \textit{$\mu$-computable ($\mu$-c.e.)} if it is computable (computably enumerable) in any representation of $\mu$.
\end{Def}

Finally, we can formally introduce the property NCR (not random w.r.t.\ any continuous measure).

\begin{Def}
  A measure $\mu$ is \textit{continuous} if every singleton has $\mu$-measure $0$.
  $X\in 2^\omega$ is \textit{NCR} if and only if $X$ is not
  \textit{$R_\mu$-$\mu$-random} w.r.t.\ any continuous probability measure $\mu$
  and any representation $R_\mu$ of $\mu$.
\end{Def}

Next, we introduce a new family of randomness tests. We will need two functions
for this, the dissipation function $g$ and the granularity $h$ of a measure.

\begin{Def}
  For any continuous probability measure $\mu$, define the \emph{granularity
    function} $g_\mu(n):=\min\{l: \forall \mid\sigma\mid=l, \mu\Cyl{\sigma}<2^{-n}\}$,
  and define the \emph{dissipation function} $h_\mu(l):=\max\{n:\forall
    \mid\sigma\mid=l, \mu\Cyl{\sigma}<2^{-n+1}\}$.
\end{Def}

We simply write $g(n)$ or $h(n)$ when the underlying measure is clear. The
function $g$ is well-defined by compactness of $2^\omega$. For any natural
number $n$, $g(n)$ gives a length $l$ by which the measure of any cylinder set
of length $l$ is less than $2^{-n}$. Given a length $l$, the dissipation
function $h(l)$ gives the binary upper bound of the measure for cylinder sets of
length $l$.

\begin{Fact} The following are some basic properties facts about $g$ and $h$.
  \label{Fact:gh}
  \begin{enumerate}[(1)]
    \item $\forall n, n<g(n)<g(n+1)<g(g(n+1))$
    \item $\forall l, h(l)\leq h(l+1)\leq h(l)+1\leq l+1$
    \item $\forall n, h(g(n))=n+1$
    \item $\lim_{l\rightarrow \infty}h(l)=\infty$
    \item $g \equiv_T h$
  \end{enumerate}
\end{Fact}
\begin{proof}
  Properties (1)-(4) follow directly from the definition or via an easy induction.

  For (5), $h(l)$ equals the largest $n$ such that $g(n-1)\leq l$, and $g(n)$ is equal to the least $l$ such that  $h(l)=n+1$, so $g \equiv_T h$.
\end{proof}

Notice that $g_\mu$ and $h_\mu$ are in general only $\mu$-c.e. But we have the following lemma, which will be useful in Section~\ref{sec:4}.
\begin{Lemma}
  \label{Lemma:g*}
  For any continuous measure $\mu$, there are $\mu$-computable, non-decreasing functions
  $\hat{h}_\mu(n),\hat{g}_\mu(n)$ such that for all $n$,
  \begin{gather*}
    h_\mu(n)\leq \hat{h}_\mu(n)\leq min\{n, h_\mu(n)+1\}, \\
    g_\mu(n)\leq \hat{g}_\mu(n)\leq g_\mu(n+1).
  \end{gather*}

\end{Lemma}
\begin{proof}
  To define $\hat{h}_\mu$, note that any representation of $\mu$ can effectively find
  an $n$ such that $2^{-n}<\mu([\sigma])<2^{-n+2}$, uniformly for any $\sigma$. Let
  $\hat{h}_\mu(l)$ be the maximum such $n\leq l$ for all $\sigma$ with length $l$.

  Now let $\hat{g}_\mu(n)$ be the minimum $l$ such that $\hat{h}_\mu(l)=n+2$. Since
  $\hat{h}_\mu \geq h_\mu$, it follows from the observation in the proof of
  Fact~\ref{Fact:gh}(5) that $\hat{g}_\mu(n) \leq g_\mu(n+1)$.

  On the other hand, by Fact \ref{Fact:gh}, we have
  $$h(\hat{g}_\mu(n))\leq \hat{h}(\hat{g}_\mu(n))=n+2.$$

  We also know $h_\mu(g_\mu(n))=n+1$, and $h_\mu$ is monotonic, so $h(\hat{g}_\mu(n))\geq g_\mu(n)$.
\end{proof}

A straightforward induction yields the following.

\begin{Corollary}
  \label{Cor:h-0}
  For the function $\hat{h}_\mu$ from Lemma~\ref{Lemma:g*}, we have that for all $l,n\in \N$, $$h^{(n)}_\mu(l)\leq \hat{h}_\mu^{(n)}(l)\leq h^{(n)}_\mu(l)+n.$$
\end{Corollary}
% \begin{proof}
% We can prove this by induction on $n$. The case $n=1$ follows directly from (6) in Fact~\ref{Fact:gh}.

% Now assume the assertion is true for $n-1$. We know $h$ and $h_0$ are non-decreasing, so we obtain
% $$h^{(n)}(l)=h(h^{(n-1)}(l))\leq h(h^{(n-1)}_0(l))\leq h^{(n)}_0(l).$$

% Similarly, we have 
% $$h^{(n)}_0(l)=h_0(h^{(n-1)}_0(l))\leq h_0(h^{(n-1)}(l)+n-1)\leq h(h^{(n-1)}(l)+n-1)+1\leq h^{(n)}(l)+n.$$
% \end{proof}

\medskip
We will now define a new randomness test. The reader should keep in mind our
main aim is to study not the random reals for a measure, but the non-random
reals. In particular, we want to devise a quantitative measure of \emph{how
  non-random} a real is.

The main difference between our test and a regular Martin-L\"{o}f test is how we
weigh cylinders. In Martin-L\"{o}f tests, we set upper bounds on the measure of
a union of cylinders. Thus, for any finite string $\sigma$, its weight is
$\mu\Cyl{\sigma}$ under measure $\mu$. When $\mu$ is Lebesgue measure, strings
with the same length would have the same weight, but this is not generally true
for other measures. However, in our new test, we assign the same weight to
strings with the same length. This means we assign a measure $\mu$ a
corresponding \emph{Hausdorff measure}. The weight of each cylinder is
determined by the dissipation function $h_\mu$. To obtain the desired
stratification, we consider iterates of $h_\mu$. The more we iterate $h_\mu$,
the slower the induced function goes to infinity, and the harder it will be to
cover reals. For technical reasons, we need to multiply by a coefficient that is
also completely determined by $h_\mu$ and the level of iteration. As mentioned
before, we will write $h$ and $\hat{h}$ for $h_\mu$ and $\hat{h}_\mu$, respectively, if
the underlying measure $\mu$ is clear.

% In a level-$1$ test, for any
% string $\sigma$, its
% weight would be $2^{-h(\mid\sigma\mid)}$, which is the uniformly binary upper bound
% for the measure of stings with length $\mid\sigma\mid$. In other words, the weight of
% $\sigma$ is the binary upper bound for all string with the same length, so
% strings with the same length receive the same weight.

% For level-$n$ tests, we just apply the $h$-function $n$ times, thus the weight
% of $\sigma$ is $2^{-h^{(n)}(\mid\sigma\mid)}$. 

\begin{Def}
  \label{Def:level}
  For any continuous measure $\mu$, a \textit{level-$n$ Solovay test for $\mu$}
  is a $\mu$-c.e. sequence $T_n$ of finite binary strings such
  that
  $$\sum_{\sigma \in T_n} (h^{(n)}(\mid\sigma\mid))^{\log n}2^{-h^{(n)}(\mid\sigma\mid)}<\infty.$$

  We say $A \in \Cant$ \emph{fails} $T_n$ if $A\in \Cyl{\sigma}$ for infinitely
  many $\sigma \in T_n$. We say $A$ is \textit{non-$\mu$-random of level $n$} if
  it fails some level-$n$ randomness test for $\mu$, and we say $A$ is
  \textit{non-$\mu$-random of level $\omega$} if it is non-$\mu$-random of level $n$ test for all
  natural numbers $n$.
\end{Def}

% \todo[color=green]{WLOG tests are infinite}
Please note that the level of a test  defined as above has nothing to do with what sometimes called the level of a Martin-L\"{o}f test (i.e., the $n$-th uniformly c.e.\ set in a Martin-L\"{o}f test).  In our definition, it is a parameter which used to measure how non-random a real is with respect to a specific continuous measure. In the following, we assume, without loss of generality, that all tests are infinite.

% Intuitively, level tests are ``length-based'' Solovay tests, the ``weight'' of each interval $[\sigma]$ depends only on the largest $\mu([\tau])$ for all $\mid\tau\mid=\mid\sigma\mid$.
If $\mu$ is Lebesgue measure, we have $h_\mu(n) = n$ and thus,
$$\sum_{\sigma \in T_n} (h^{(n)}(\mid\sigma\mid))^{\log
      n}2^{-h^{(n)}(\mid\sigma\mid)}=\sum_{\sigma \in T_n}\mid\sigma\mid^{\log n}2^{-\mid\sigma\mid},$$
so in this case a level-$1$ Solovay test coincides with the standard notion of a Solovay test\cite[6.2.7]{downey2010algorithmic}

\medskip
We next establish some basic properties of the new test notion.
% We omit the measure notion as subscripts when the intend measure is clear or does not matter.
The following Lemma follows easily by analyzing the derivative.
\begin{Lemma}
  \label{Le:mono}
  The function $f(x):=x^{\log n}2^{-x}$ is decreasing to 0 for $x>\log n$ as $x$ goes to infinity.
\end{Lemma}

We first show that $\mu$-computable reals are non-$\mu$ random of level $\omega$.

\begin{Proposition} \label{prop:computable_omega}
  If a real $A$ is computable in $\mu$, then $A$ is non-$\mu$ random of level $\omega$ for all continuous measures $\mu$.
\end{Proposition}

\begin{proof}
  If $A$ is a $\mu$-computable real, then we can compute arbitrary long initial segments of $A$ from any representation of $\mu$. By Fact \ref{Fact:gh}(2) and Lemma~\ref{Lemma:g*}, the $\mu$-computable function $\hat{h}(l)$ is non-decreasing, $h(l)\leq \hat{h}(l)\leq h(l)+1$, and $\lim_{l\rightarrow \infty}\hat{h}(l)$ and $\lim_{l\rightarrow \infty}h(l)$ are both infinite. Then for any natural number $n$, if $\sigma$ is an initial segment of $A$ and $\hat{h}^{(n)}(\mid\sigma\mid)$ is greater than $n+\log n$, by Lemma \ref{Le:mono} and Corollary \ref{Cor:h-0}, we have  the following inequality:
  $$(h^{(n)}(\mid\sigma\mid))^{\log n}2^{-h^{(n)}(\mid\sigma\mid)}\leq (\hat{h}^{(n)}(\mid\sigma\mid)-n)^{\log n}2^{-\hat{h}^{(n)}(\mid\sigma\mid)+n}.$$

  So, for fixed $n$, let $\{\sigma_i\}$ be a $\mu$-computable sequence of initial segments of $A$ such that the following two inequalities are satisfied, for all $i \in \omega$:
  \begin{gather*}
    (\hat{h})^{(n)}(\mid\sigma_i\mid)>n+\log n, \\
    (\hat{h}^{(n)}(\mid\sigma_i\mid)-n)^{\log n}2^{-\hat{h}^{(n)}(\mid\sigma_i\mid)+n}<2^{-i}.
  \end{gather*}
  Then $\{\sigma_i\}_{i \in \N}$ is a level-$n$ test which covers $A$. Therefore, $A$ is non-$\mu$ random of level $\omega$.
\end{proof}

The next proposition shows the relation between level tests and  Martin-L\"{o}f tests.

\begin{Proposition}
  If a real $A$ is non-$\mu$-random of level $1$, then $A$ is not $\mu$-Martin-L\"{o}f random.
\end{Proposition}

\begin{proof}
  If $n=1$, the sum in Definition~\ref{Def:level} becomes
  \[
    \sum_{\sigma\in T_1} 2^{-h(\mid\sigma_i\mid)}.
  \]
  By the definition of $h$, we have $\mu\Cyl{\sigma} < 2^{-h(\mid\sigma\mid)+1}$, thus any level-$1$ test is a standard Solovay test.
  Moreover, for a probability measure, any real covered by a Solovay test is also
  covered by a Martin-Löf test, see for example~\cite[Theorem~6.2.8]{downey2010algorithmic}.
\end{proof}

Next, we show that the level tests are indeed nested.

\begin{Proposition}
  Every level-$n$ test is also a level-$(n-1)$ test.
\end{Proposition}
\begin{proof}
  Assume $\{\sigma_i\}_{i\in \N}$ is a level-$n$ test. By Fact \ref{Fact:gh}(2),
  $$h^{(n-1)}(\mid\sigma_i\mid) \geq h^{(n)}(\mid\sigma_i\mid),$$
  for all $i$. Moreover, $\mid\sigma_i\mid \to \infty$ as $i \to \infty$ since $\{\sigma_i\}$ is a level-$n$ test. By \ref{Fact:gh}(4), this implies that, for all but finitely many $i$,
  $$h^{(n)}(\mid\sigma_i\mid) > \log (n-1).$$

  By Lemma \ref{Le:mono} and the inequalities above, for all but finitely many $i$, the following holds: $$(h^{(n-1)}(\mid\sigma_i\mid))^{\log (n-1)}2^{-h^{(n-1)}(\mid\sigma_i\mid)}<(h^{(n)}(\mid\sigma_i\mid))^{\log( n-1)}2^{-h^{(n)}(\mid\sigma_i\mid)}.$$

  Furthermore, we know $h^{(n)}(\mid\sigma_i\mid)$ is positive and $\log (n-1)<\log n$, so we have
  $$(h^{(n)}(\mid\sigma_i\mid))^{\log (n-1)}2^{-h^{(n)}(\mid\sigma_i\mid)}<(h^{(n)}(\mid\sigma_i\mid))^{\log n}2^{-h^{(n)}(\mid\sigma_i\mid)}.$$

  Finally, since $\{\sigma_i\}_{i\in \N}$ is an level-$n$ test,
  \[
    \sum_{i\in \N} (h^{(n-1)}(\mid\sigma_i\mid))^{\log (n-1)}2^{-h^{(n-1)}(\mid\sigma_i\mid)}<\sum_{i\in \N} (h^{(n)}(\mid\sigma_i\mid))^{\log n}2^{-h^{(n)}(\mid\sigma_i\mid)}<\infty.
  \]
  So  $\{\sigma_i\}_{i\in \N}$ is also a level-$(n-1)$ test.
\end{proof}

The previous results justify thinking of level tests as a hierarchy of
non-randomness for continuous measures.
In particular, we have

\begin{center}
  X is non-$\mu$ random of level $\omega$\\[1ex]
  $\big\Downarrow$ \\[1ex]
  X is non-$\mu$ random of level $n+1$\\[1ex]
  $\big\Downarrow$ \\[1ex]
  X is non-$\mu$ random of level $n$\\[1ex]
  $\big\Downarrow$ \\[1ex]
  X is not $\mu$-random.
\end{center}

\bigskip
It is not too hard to construct a measure for which this hierarchy is
proper (see~\cite{li:thesis}), while for other measures (such as Lebesgue
measure on $\Cant$) it collapses.

One can define a similar hierarchy for NCR instead of for individual measures,
saying that a real $X\in 2^\omega$ is \emph{NCR of level $n$ ($\omega$)} if and only
if $X$ is non-$\mu$ random of level $n$ ($\omega$) for every continuous probability
measure $\mu$. Interestingly, this hierarchy for NCR overall collapses, mostly
due to the correspondence between continuous measures and Hausdorff measures
established by \emph{Frostman's Lemma} (see \cite{reimann2008effectively}).
This is shown in~\cite{li:thesis}.

\section{Constructing non-random r.e.a.\ reals}~
\label{sec:3}

The goal of this section is to construct level-$n$ non-random reals that are r.e.a.\ a given level-$2n$ non-random real $A$. In fact, we can construct such a real in any Turing degree r.e.a.\ $A$.

To this end, we first introduce a general construction technique which builds a real $C$ r.e.a.\  a given real $A$.

The basic idea is to add a large amount of ``1''s between each bit of $B$, where the number of ``1''s is still computable by $B$.

\begin{Construction}
  \label{Cons:1}
  Assume for a given $A$ and a real $B$ r.e.\ above $A$, we have $W^A_e=B$ for some $e$. Without loss of generality, we may assume the first bit of $B$ is ``1"
  and it takes $\Phi^A_e$ only one step to halt on input ``0" with no use of the oracle. We also assume that $B$ is infinite.

  Denote the $i$-th bit of $A$ by $a_i$ and the $i$-th bit of $B$ by $b_i$. By our assumption, $b_0=1$.

  Let
  \[
    m_i = min\{j>i \colon \Phi^A_{e}(j)\downarrow\},
  \]
  that is, $m_i$ is the least element of $B$ which is greater than $i$. Define the function $f: \N\rightarrow\N$ as
  \begin{equation*}
    f(i) = \begin{cases}
      min\{s\mid\forall j\leq m_i(\Phi^A_{e}(j)\downarrow\implies \Phi^A_{e,s}(j)\downarrow)\} & \text{ if } i \in B,   \\
      1                                                                                        & \text{ if } i\notin B.
    \end{cases}
  \end{equation*}

  When $i\in B$, $f(i)$ is the minimum  number such that for all $j\leq m_i$ and
  $j\in B$, $\Phi^A_e(j)$ halts within $f(i)$ many steps. Since $A \leq_T B$,  $f$ is
  $B$-computable. Define a sequence of finite binary strings $C_i$ as follows:
  $$C_i =b^{f(0)}_0 \Conc 0\Conc b^{f(1)}_1\Conc 0\Conc b^{f(2)}_2\Conc 0\Conc \ldots \Conc b^{f(i)}_i.$$
  Let $C=\lim_i C_i$. Since $b_i$ and  $f(i)$ are $B$-computable, so is $C$. On
  the other hand, the first $i$ bits of $B$ are coded in $C_i$: Each block of ones
  corresponds to exactly one element in $B$ less than $i$. Therefore, $C\equiv_T B$.
\end{Construction}

We illustrate Construction \ref{Cons:1} with an example. Let $A$ be a real and
$B=W^A_e$ as in Construction \ref{Cons:1} and let $s_A(n)$ be the settling time of
$\Phi^A_e(n)$. Assume the first few values of $B$ and $s_A$ are as given in the
following table.

\bigskip
\begin{center}
  \begin{tabular}{ c || c | c | c | c | c | c }
    $n$        & $0$                     & $1$                     & $2$                   & $3$                     & $4$                     & \mbox{ } \dots \mbox{ } \\
    \hline
    $\Phi^A_e$ & $\Phi^A_e(0)\downarrow$ & $\Phi^A_e(1)\downarrow$ & $\Phi^A_e(2)\uparrow$ & $\Phi^A_e(3)\downarrow$ & $\Phi^A_e(4)\downarrow$ & \mbox{ } \dots \mbox{ } \\
    $s_A$      & 1                       & 37                      & $\infty$              & 134                     & 28                      & \mbox{ } \dots \mbox{ } \\
  \end{tabular}
\end{center}

\bigskip
Following Construction \ref{Cons:1}, we obtain the first few bits of $C$ as follows.

\bigskip
\begin{center}
  \begin{tabular}{ c || c | c | c | c | c | c }
    $n$ & $0$                                    & $1$                                     & $2$          & $3$                                     & $4$                     & \mbox{ } \dots \mbox{ } \\
    \hline
    B   & 1$^\frown$                             & 1$^\frown$                              & 0$^\frown$   & 1$^\frown$                              & 1$^\frown$              & \mbox{ } \dots \mbox{ } \\
    f   & 37                                     & 134                                     & 1            & 134                                     & \mbox{ } \dots \mbox{ } & \mbox{ } \dots \mbox{ } \\
        & $\Downarrow$                           & $\Downarrow$                            & $\Downarrow$ & $\Downarrow$                            & $\Downarrow$                                      \\
    C   & $\underbrace{1...1}_\text{37}0^\frown$ & $\underbrace{1...1}_\text{134}0^\frown$ & $00^\frown$  & $\underbrace{1...1}_\text{134}0^\frown$ & $1 \dots$               & \mbox{ } \dots \mbox{ } \\
  \end{tabular}
\end{center}

\bigskip
We now show that non-randomness properties of $A$ carry over to $C$. Intuitively,
if we know $\sigma$ is an initial segment of $A$, we can use it to
``approximate" some initial segment of $B$ by waiting for
$\Phi^\sigma_e(*)$ to converge,
until the use exceeds $\sigma$. But we cannot effectively get
any initial segment of $B$ in this way, as we have no upper bound
on the settling time of $\Phi^\sigma_e$, therefore we cannot find a effective
cover of $B$ by using this approximation.

We address this problem in the construction of $C$ by adding long series of ones, thereby decreasing the cost in measure of adding an incorrect string to a test. Consider the case when we use a long enough initial segment of
$A$ to approximate the first $n$ bits of $B$ for $s$ steps, but the
approximation $\tau$ we got for $B$ turns out to be wrong. Let $m$ be the index of the first incorrect bit. Then the settling time of $\Phi^\sigma_e(m)$ must be greater than $s$. By
Construction \ref{Cons:1}, an initial segment of $C$ is of the form
$$b^{f(0)}_0 \Conc 0\Conc b^{f(1)}_1\Conc 0\Conc b^{f(2)}_2\Conc 0\Conc
  \ldots \Conc \underbrace{111\ldots 1}_{\text{more than s}}.$$

By picking a large $s$,
the total measure of all possible strings of the above form is small. Eventually,
we can effectively find a cover of $C$ from any initial segment of $A$.
\begin{Thm}
  \label{Thm:2n-n}
  For any continuous measure $\mu$, if $A$ is non-$\mu$ random of level $2n$, $B$ is r.e.a.\ $A$, and $C$ is obtained from $B$ via Construction~\ref{Cons:1}, then $C$ is non-$\mu$ random of level $n$.
\end{Thm}
\begin{proof}
  We define an auxiliary function $t$ from $2^{<\omega}\x \N$ to finite subsets of $2^{<\omega}$:
  \begin{equation*}
    t(\sigma,n) := \begin{cases}
      \{\sigma\}                                                                             & \text{if $\mid\sigma\mid<n$;}     \\
      \{\sigma\uhr_n\}\cup \bigcup_{i=0}^{n} \{\sigma\!\uhr_i \!\Conc 1^{\mid\sigma\mid-i}\} & \text{if $\mid\sigma\mid\geq n$.}
    \end{cases}
  \end{equation*}
  % When $\mid\sigma\mid<n$, $t(\sigma,n)$ represents the singleton containing only $\sigma$, when $\mid\sigma\mid\geq n$, $t(\sigma,n)$ represent the set consists of $\sigma\uhr_n$ and initial segments of $\sigma$ of length $i\leq n$ joint with a sequence of "$1$" up to the length of $\mid\sigma\mid$.

  \begin{Lemma}
    \label{Lemma:2n-n}
    If $\{\sigma_i\}_{i\in \N}$ is a level-$2n$ randomness test of $\mu$, then $$\bigcup_{i\in \N} t(\sigma_i,\hat{h}^{(n)}(\mid\sigma_i\mid))$$ is a level-$n$ randomness test of $\mu$.
  \end{Lemma}
  \begin{proof}[Proof of Lemma \ref{Lemma:2n-n}]
    By Fact \ref{Fact:gh} (4) and Lemma~\ref{Lemma:g*}, we have $$\lim_n \hat{h}(n)\rightarrow \infty.$$ Hence, for fixed $n$ it holds that for all but finitely many $i$, $$\hat{h}^{(2n)}(\mid\sigma_i\mid)> \log 2n+2n.$$

    Fact \ref{Fact:gh} and Lemma~\ref{Lemma:g*} also imply that
    $$\hat{h}^{(n)}(\mid\sigma_i\mid)\leq \mid\sigma_i\mid.$$

    Therefore, for all $i$,
    $$t(\sigma_i,\hat{h}^{(n)}(\mid\sigma_i\mid))=\{\sigma_i\uhr_{\hat{h}^{(n)}(\mid\sigma_i\mid)}\}\cup \bigcup_{j=0}^{\hat{h}^{(n)}(\mid\sigma_i\mid)} \{\sigma_i\uhr_j\!\Conc 1^{\mid\sigma_i\mid-j}\}.$$

    The contribution of $\sigma_i\uhr_{\hat{h}^{(n)}(\mid\sigma_i\mid)}$ to a level-$n$ test is
    \begin{equation*}
      \begin{split}
        & (h^{(n)}(\mid\sigma_i\uhr_{\hat{h}^{(n)}(\mid\sigma_i\mid)}\mid))^{\log n}2^{-h^{(n)}(\mid\sigma_i\uhr_{\hat{h}^{(n)}(\mid\sigma_i\mid)}\mid)}\\
        & \quad = (h^{(n)}({\hat{h}^{(n)}(\mid\sigma_i\mid)}))^{\log n}2^{-(h^{(n)}({\hat{h}^{(n)}(\mid\sigma_i\mid)}))}.
      \end{split}
    \end{equation*}

    By Lemma \ref{Le:mono}, for all but finitely many $i$,
    \begin{equation}
      \begin{split}
        & (h^{(n)}({\hat{h}^{(n)}(\mid\sigma_i\mid)}))^{\log n}2^{-(h^{(n)}({\hat{h}^{(n)}(\mid\sigma_i\mid)}))} \\
        & \quad \leq \; (h^{(2n)}(\mid\sigma_i\mid))^{\log n}2^{-h^{(2n)}(\mid\sigma_i\mid)}\\
        & \quad \leq \; (h^{(2n)}(\mid\sigma_i\mid))^{\log 2n}2^{-h^{(2n)}(\mid\sigma_i\mid)}.
      \end{split}
      \tag{*} \label{equ1}
    \end{equation}

    Moreover, the contribution of $\bigcup_{j=0}^{\hat{h}^{(n)}(\mid\sigma_i\mid)} \{\sigma_i\uhr_j\!\Conc 1^{\mid\sigma_i\mid-j}\}$ to a level-$n$ test is
    \begin{equation*}
      \begin{split}
        & \sum_{j=0}^{\hat{h}^{(n)}(\mid\sigma_i\mid)}(h^{(n)}(\mid\sigma_i\uhr_j\!\Conc 1^{\mid\sigma_i\mid-j}\mid))^{\log n}2^{-h^{(n)}(\mid\sigma_i\uhr_j\!\Conc 1^{\mid\sigma_i\mid-j}\mid)}\\
        & \quad =\;(\hat{h}^{(n)}(\mid\sigma_i\mid)+1)((h^{(n)}(\mid\sigma_i\mid))^{\log n}2^{-h^{(n)}(\mid\sigma_i\mid)})\\
      \end{split}
    \end{equation*}

    By Corollary~\ref{Cor:h-0}, for all but finitely many $i$, we have $$(\hat{h}^{(n)}(\mid\sigma_i\mid)+1)<2\cdot h^{(n)}(\mid\sigma_i\mid).$$

    Therefore
    \begin{equation*}
      \begin{split}
        &(\hat{h}^{(n)}(\mid\sigma_i\mid)+1)((h^{(n)}(\mid\sigma_i\mid))^{\log n}2^{-h^{(n)}(\mid\sigma_i\mid)})\\
        & \quad \leq \;  2\cdot h^{(n)}(\mid\sigma_i\mid)((h^{(n)}(\mid\sigma_i\mid))^{\log n}2^{-h^{(n)}(\mid\sigma_i\mid)})\\
        & \quad = \;2\cdot (h^{(n)}(\mid\sigma_i\mid))^{\log 2n}2^{-h^{(n)}(\mid\sigma_i\mid)}.
      \end{split}
    \end{equation*}

    By Fact \ref{Fact:gh}, $h^{(n)}(\mid\sigma_i\mid)\geq h^{(2n)}(\mid\sigma_i\mid)$ and $\lim_i h(i)=\infty$. Together with Lemma \ref{Le:mono}, for all but finitely many $\sigma_i$, we have the following upper bound.

    \begin{equation}
      \begin{split}
        & 2\cdot (h^{(n)}(\mid\sigma_i\mid))^{\log 2n}2^{-h^{(n)}(\mid\sigma_i\mid)}\\
        & \quad \leq \; 2\cdot (h^{(2n)}(\mid\sigma_i\mid))^{\log 2n}2^{-h^{(2n)}(\mid\sigma_i\mid)}.
      \end{split}
      \tag{**} \label{equ2}
    \end{equation}

    Together, equations \eqref{equ1} and \eqref{equ2} yield the following upper
    bound for the contribution of $t(\sigma_i,\hat{h}^{(n)}(\mid\sigma_i\mid))$ to a
    level-$n$ test:
    \begin{equation*}
      \begin{split}
        &(h^{(n)}(\mid\sigma_i\uhr_{\hat{h}^{(n)}(\mid\sigma_i\mid)}\mid))^{\log n}2^{-h^{(n)}(\mid\sigma_i\uhr_{\hat{h}^{(n)}(\mid\sigma_i\mid)}\mid)} \\
        & \qquad \qquad + \; \sum_{j=0}^{\hat{h}^{(n)}(\mid\sigma_i\mid)}(h^{(n)}(\mid\sigma_i\uhr_j\!\Conc 1^{\mid\sigma_i\mid-j}\mid))^{\log n}2^{-h^{(n)}(\mid\sigma_i\uhr_j\!\Conc 1^{\mid\sigma_i\mid-j}\mid)}
        \\
        & \quad \leq \; (h^{(2n)}(\mid\sigma_i\mid))^{\log 2n}2^{-h^{(2n)}(\mid\sigma_i\mid)} + 2\cdot (h^{(2n)}(\mid\sigma_i\mid))^{\log 2n}2^{-h^{(2n)}(\mid\sigma_i\mid)}\\
        & \quad \leq \;  3\cdot(h^{(2n)}(\mid\sigma_i\mid))^{\log 2n}2^{-h^{(2n)}(\mid\sigma_i\mid)}.
      \end{split}
    \end{equation*}

    Hence if $\{\sigma_i\}_{i\in \N}$ is a level-$2n$ test,  $\bigcup_{i\in \N} t(\sigma_i,h_0^{(n)}(\mid\sigma_i\mid))$ is a level-$n$ test.
  \end{proof}

  \medskip
  We continue the proof of Theorem~\ref{Thm:2n-n}.
  Assume $\{\sigma_i\}_{i\in \N}$ is a level-$2n$ test that $A$ fails. For each $i$, consider the set $W^{\sigma_i}_{e,\mid\sigma_i\mid}$.
  Write the characteristic sequence of $W^{\sigma_i}_{e,\mid\sigma_i\mid}$ as $b_{i,0}\, b_{i,1}\, b_{i,2}\, ...\, b_{i,\mid\sigma_i\mid}$, and put $b_{i,\mid\sigma_i\mid+1}=1$ for convenience. For $k\leq \mid\sigma_i\mid$, define $m_{i,k}:=\min\{j>k\mid b_{i,j}=1\}$, and define the function $f_i:\{1,2,3,...,\mid\sigma_i\mid\}\rightarrow\N$ as
  \begin{equation*}
    f_i(k) = \begin{cases}
      1                                                                             & \text{if $b_{i,k}=0$;}                                          \\
      \min\{l\mid\forall j\leq m_{i,k}(b_{i,j}=1\implies W^{\sigma_i}_{e,l}(j)=1)\} & \text{if$(b_{i,k}=1) \land (m_{i,k}\neq \mid\sigma_i\mid+1)$; } \\
      \mid\sigma_i\mid                                                              & \text{if$(b_{i,k}=1) \land (m_{i,k}=\mid\sigma_i\mid+1)$. }
    \end{cases}
  \end{equation*}

  Lastly, define $$\tau_i =b_{i,0}^{f_i(0)}\Conc 0\Conc b_{i,1}^{f_i(1)}\Conc 0\Conc b_{i,2}^{f_i(2)}\Conc 0\Conc \ldots \Conc b_{i,\mid\sigma_i\mid}^{f_i(\mid\sigma_i\mid)}\uhr_{\mid\sigma_i\mid}.$$

  Since $\mid\tau_i\mid=\mid\sigma_i\mid$, $\{\tau_i\}_{i\in \N}$ is also a level-$2n$ test. By Lemma~\ref{Lemma:2n-n}, $\bigcup_{i\in \N} t(\tau_i,\hat{h}^{(n)}(\mid\tau_i\mid))$ is a level-$n$ test.

  \medskip
  \textbf{Claim:} $C$ fails the test $\bigcup_{i\in \N} t(\tau_i,\hat{h}^{(n)}(\mid\tau_i\mid))$.

  \medskip
  % Let $\{\sigma_i\}$ be a level-$2n$ test that $A$ fails, and let $\{\tau_i\}$ be defined as above. 
  We will show that if $\sigma_i\sqsubset A$, $t(\tau_i,\hat{h}^{(n)}(\mid\tau_i\mid))$ contains an initial segment of $C$.

  By the assumption on $B$ in Construction~\ref{Cons:1}, we have $b_{i,0}=1$ for all $i$. Since we assume $\sigma_i\sqsubset A$, it follows that for any $a\leq \mid\sigma_i\mid$, $b_{i,a}=1$ implies $b_a=1$.

  If $\tau_i\uhr_{\hat{h}^{(n)}(\mid\tau_i\mid)}$ is an initial segment of $C$, then by the definition of $t$, $C$ trivially fails the test.
  So let us assume $\tau_i\uhr_{\hat{h}^{(n)}(\mid\tau_i\mid)}$ is not an initial segment of $C$. Define
  $$k_i=\max\{l\mid\forall j< l(b_{i,j}=b_j)\land (b_{i,l}=1)\}.$$

  Thus, $k_i$ is the maximal length for which $b_{i,k_i}=1$ and
  $$b_0b_1b_2...b_{k_i-1}=b_{i,0}b_{i,1}b_{i,2}...b_{i,k_i-1}.$$

  Then for any $k<k_i$, by the definition of $f_i$, we have $f_i(k)=f(k)$. As we assumed $\tau_i\uhr_{\hat{h}^{(n)}(\mid\tau_i\mid)}$ is not an initial segment of $C$, by comparing lengths, we know that
  $$k_i<\hat{h}^{(n)}(\mid\tau_i\mid).$$

  Let $j$ be the minimum number such that $b_j\neq b_{i,j}$, thus $b_j=1$, $b_{i,j}=0$ and $k_i<j<\hat{h}^{(n)}(\mid\tau_i\mid)$. We have that
  $$\Phi^{A}_{e,\mid\sigma_i\mid}(j)=\Phi^{\sigma_i}_{e,\mid\sigma_i\mid}(j)=b_{i,j}=0$$
  $$\Phi^{A}_{e,f(k_i)}(j)=b_{j}=1.$$

  This means $f(k_i)\geq \mid\sigma_i\mid$, so we can find an element of $t(\tau_i,\hat{h}^{(n)}(\mid\tau_i\mid))$ which is also an initial segment of $C$ as follows.
  \begin{equation*}
    \begin{split}
      & \tau_i\uhr_{\Sigma_{t=0}^{k_i-1}(f_i(t)+1)}\Conc 1^{\mid\sigma_i\mid-\Sigma_{t=0}^{k_i-1}(f_i(t)+1)}\\
      & \quad =b_{i,0}^{f_i(0)}\Conc 0\Conc b_{i,1}^{f_i(1)}\Conc 0\Conc \ldots \Conc b_{i,k_i-1}^{f_i(k_i-1)}\Conc 0\Conc 1^{\mid\sigma_i\mid-\Sigma_{t=0}^{k_i-1}(f_i(t)+1)}\\
      & \quad \sqsubset b_{i,0}^{f_i(0)}\Conc 0\Conc b_{i,1}^{f_i(1)}\Conc 0\Conc \ldots \Conc b_{i,k_i-1}^{f_i(k_i-1)}\Conc 0\Conc 1^{\mid\sigma_i\mid}\\
      & \quad \sqsubset b_{i,0}^{f_i(0)}\Conc 0\Conc b_{i,1}^{f_i(1)}\Conc 0\Conc \ldots \Conc b_{i,k_i-1}^{f_i(k_i-1)}\Conc 0\Conc 1^{f(k_i)}\\
      & \quad = b^{f(0)}_0\Conc 0\Conc b^{f(1)}_1\Conc 0\Conc b^{f(2)}_2\Conc 0\Conc\ldots \Conc b^{f(k_i-1)}_{k_i-1}\Conc 0\Conc b_{k_i}^{f(k_i)}\\
      &\quad \sqsubset C.
    \end{split}
  \end{equation*}

  It follows that $C$ is covered by the level-$n$ test $\bigcup_{i\in \N} t(\tau_i,\hat{h}^{(n)}(\mid\tau_i\mid))$ and therefore non-$\mu$-random of level $n$. This completes the proof of Theorem~\ref{Thm:2n-n}.
\end{proof}

\section{Constructing non-random reals using a self-modulus}~
\label{sec:4}

We begin this section by reviewing the concepts of modulus and self-modulus.

\begin{Def}
  For any function $f,g:\N\rightarrow\N$, we say $f$ \textit{dominates} $g$ if $f(n)>g(n)$ for all but finitely many $n\in \N$. For any real $A$, we say a function $f$ is a \textit{modulus (of computation)} for $A$ if every function dominating $f$ can compute $A$. We say A has a \textit{self-modulus} if there is a modulus $f_A$ of $A$ such that  $f_A\equiv_T A$.
\end{Def}

Arguably the best-known class of reals with a self-modulus is $\Delta^0_2$, see,
for example,~\cite[Theorem~5.6.6]{soare2016turing}.

% \begin{Proposition}[Folklore]
% \label{Folklore}
% Every $\Delta^0_2$ real $A$ has a self-modulus.
% \end{Proposition}
% \begin{proof}
% By Shoenfield's Limit Lemma, we know there is a total recursive function $f(n,k)$ such that 
% $$\lim_{k\rightarrow \infty} f(n,k) = A(n).$$ 

% Define $t(n)$ as the minimum $k>n$ such that $f(m,k)=A(m)$ for any $m\leq n$. It holds that $t \leq_T A$. Now assume a function $s$ dominates $t$. We may assume $s(n) > t(n)$ for \emph{all} $n$ and that $s$ is increasing (since $s$ can compute such a function).

% We show that $s$ can compute $A$. For $i\in \N$ consider sets 
% $$M_i = \{f(n,k) \colon s^{(i)}(n)\leq k\leq s^{(i+1)}(n)\}.$$

% We have 
% $$s^{(i+1)}(n)> t(s^{(i)}(n))> s^{(i)}(n)> n.$$ 

% By the definition of $t$, for any $i$, $M_i$ always contains the correct $n$-th bit of $A$, and $\lim_{k\rightarrow \infty} f(n,k) = A(n)$ implies for some large enough number $i$, $M_i$ is a singleton. Hence we can compute $A(n)$ by checking the set  $M_i$ for $i\in \N$ until it is a singleton for some $i$.
% \end{proof}

\medskip
Our second construction method will take real $A$ with a self-modulus $f_A$ and define another real $B\equiv_T A$.

\begin{Construction}~
  \label{Cons:2}
  Assume $A=a_0 \, a_1 \, a_2 \, a_3 \dots$ and $f_A\equiv_T A$ is a self-modulus of $A$. Without loss of generality, we can assume $f_A(n)$ is increasing.

  We define our first string $B_0$ as
  $$B_0 =1^{f_A(0)}\Conc 0\Conc a_0,$$

  and inductively put
  $$B_{n+1} =B_n\Conc 1^{f_A(\mid B_n\mid)}\Conc 0\Conc a_{n+1}.$$

  Let
  $$B =\lim_{i\rightarrow \infty} B_i.$$

  In the following, $l_n$ will denote the length of $B_n$.

  \bigskip
  As each $a_i$ is coded into $B_i$ immediately following a block of the form $1^{f_A(\mid B_i\mid)}\Conc 0$, it follows that that $A\leq_T B$. Since the $B_i$ are uniformly computable in $A$, $B\leq_T A$. Therefore, $B\equiv_T A$.
\end{Construction}

We have the following property of Construction \ref{Cons:2}.

\begin{Thm}
  \label{Thm:modulelevelinfty}
  If $A$ has a self-modulus $f_A$ and $B$ is defined from $A$ and $f_A$ as in
  Construction~\ref{Cons:2}, then $B$ is non-$\mu$ random of level $\omega$ for
  any continuous $\mu$.
\end{Thm}

\begin{proof}
  Let $\mu$ be a continuous measure. If there is a $\mu$-computable function
  dominating $f_A$, then $\mu$ can compute $B$ as well as $A$, so $B$ is not
  $\mu$-random of level $\omega$. Therefore, let us assume there is no
  $\mu$-computable function dominating $f_A$. As before, we write $g$ and $h$ to
  denote the granularity and dissipation function $g_\mu$ and $h_\mu$,
  respectively.

  \begin{Lemma}
    \label{Le:moduleDominite}
    If there is no $\mu$-computable function dominating $f_A$, then for any $k\in
      \N$, there are infinitely many $n$ such that $\hat{g}^{(k)}(2l_n+1)<f_A(l_n)$,
    where $l_n$ is the length of $B_n$ as defined in Construction~\ref{Cons:2} and
    $\hat{g}$ is as defined in Lemma \ref{Lemma:g*}.
  \end{Lemma}

  \begin{proof}[Proof of Lemma \ref{Le:moduleDominite}]
    Suppose for a contradiction there is an $n_0$ such that for any $m>n_0$, it holds that
    $$\hat{g}^{(k)}(2l_m+1)>f_A(l_m).$$

    Define a function $G$ as follows. Put
    $G(0) ={\hat{g}}^{(k)}(2l_{n_0}+1)$ and inductively define
    $G(i+1) =G(i)+{\hat{g}}^{(k)}(2G(i)+1)+2$. Since $\hat{g}$
    is computable in $\mu$, $G\leq_T \mu$.

    We claim that $G(i) \geq l_i$ for $i\geq n_0$.
    For $i=n_0$, $$G(n_0)> G(0)=\hat{g}^{(k)}(2l_{n_0}+1)\geq l_{n_0}.$$

    For $i \geq n_0$, if $G(i)>l_i$,
    \[
      \begin{split}
        G(i+1) & =G(i)+\hat{g}^{(k)}(2G(i)+1)+2 \\
        & >l_i+\hat{g}^{(k)}(2l_i+1)+2>l_i+f_A(l_i)+2=l_{i+1}.
      \end{split}
    \]
    So $G(i) \geq l_i$ for $i\geq n_0$.
    Moreover, by the definition of $B_i$, $l_i > f_A(i)$ for all $i$.

    % Firstly, we have $l_0=f_A(0)+2>f_A(0)$. 

    % Next, if $l_i>f_A(i)$, then $l_{i+1}=l_i+f_A(l_i)+2>f_A(i+1)$. So $l_i$ dominates $f_A(i)$. \\

    Combining the previous two facts, we obtain a $\mu$-computable function $G$ such that
    $G(i) \geq f_A(i)$ for $i \geq n_0$, contradicting the assumption that there is no $\mu$-computable function dominating $f_A$. So there are infinitely many $n$ such that  $${\hat{g}}^{(k)}(2l_n+1)<f_A(l_n).$$
  \end{proof}

  To complete the proof of Theorem~\ref{Thm:modulelevelinfty}, for any $k\in\N$,  we define the following set of strings:
  $$T_k =\{\sigma\Conc 1^{\hat{g}^{(k)}(2\mid\sigma\mid)}\mid\sigma\in 2^{<\omega}\}.$$

  Then
  \begin{equation*}
    \begin{split}
      & \sum_{\tau\in T_k} (h^{(k)}(\mid\tau\mid))^{\log k}2^{-h^{(k)}(\mid\tau\mid)} \\
      & \quad =\sum_{i=0}^\infty 2^i(h^{(k)}(i+\hat{g}^{(k)}(2i)))^{\log k}2^{-h^{(k)}(i+\hat{g}^{(k)}(2i))} \\
      & \quad =\sum_{i> \log k} 2^i(h^{(k)}(i+\hat{g}^{(k)}(2i)))^{\log k}2^{-h^{(k)}(i+\hat{g}^{(k)}(2i))} + \gamma_k,
    \end{split}
  \end{equation*}
  where
  \[
    \gamma_k = \sum_{i \leq \log k} 2^i(h^{(k)}(i+\hat{g}^{(k)}(2i)))^{\log k}2^{-h^{(k)}(i+\hat{g}^{(k)}(2i))} < \infty.
  \]

  Moreover, by Fact \ref{Fact:gh} and Lemma \ref{Lemma:g*},
  $$h^{(k)}(i+\hat{g}^{(k)}(2i)) \geq h^{(k)}(\hat{g}^{(k)}(2i))\geq h^{(k)}(g^{(k)}(2i))\geq 2i.$$

  By Lemma \ref{Le:mono}, we have
  \begin{align*}
    \begin{split}
      &\sum_{i> \log k} 2^i(h^{(k)}(i+\hat{g}^{(k)}(2i)))^{\log k}2^{-h_{\mu}^{(k)}(i+\hat{g}^{(k)}(2i))} + \gamma_k \\
      & \quad \leq \sum_{i> \log k}^\infty 2^i(2i)^{\log k}2^{-2i}+ \gamma_k\\
      & \quad = \sum_{i> \log k} (2i)^{\log k} 2^{-i}+ \gamma_k <\infty.\\
    \end{split}
  \end{align*}

  Thus, $T_k$ is a level-$k$ test. Finally, when $\hat{g}^{(k)}(2l_n+1)<f_A(l_n)$, we have
  $$B_n\Conc 1^{\hat{g}^{(k)}(2l_n)}\sqsubset B_n\Conc 1^{f_A(l_n)}\sqsubset B.$$

  By the definition of $T_k$, any string of the form $B_n\Conc
    1^{\hat{g}^{(k)}(2l_n)}$ is in $T_k$.  By Lemma~\ref{Le:moduleDominite}, for any
  $k$, $\hat{g}^{(k)}(2l_n+1)<f_A(l_n)$ is true for infinitely many $n$.
  Therefore, $B$  fails $T_k$. Since $k$ was arbitrary, $B$ is non-$\mu$-random of level $\omega$.
\end{proof}

\section{Turing degrees of NCR Reals}~
\label{sec:5}

Using the constructions presented in the previous two sections, we exhibit a large class of Turing degrees that contain NCR elements, as formulated in the Introduction.
% We actually show something slightly stronger, as we can obtain NCR elements of arbitrarily high non-randomness level. 

% \begin{Thm}
% \mbox{}
% \begin{enumerate}[(a)]
% 	\item Any Turing degree r.e.a.\ a real NCR of level $2n$ contains a real that is NCR of level $n$.
% 	\item Any degree with a self-modulus contains a real that is NCR of level $\omega$.
% \end{enumerate}
% \end{Thm}
% \begin{proof}
% Part (a) follows Theorem~\ref{Thm:2n-n}, Part (b) follows from Theorem~\ref{Thm:modulelevelinfty}.
% \end{proof}

\begin{Def}
  A real is \textit{1-REA} if it is recursively enumerable. A real is \textit{$(n+1)$-REA} if it is r.e.a.\ some $n$-REA real. A Turing degree is $n$-REA if it contains an $n$-REA real.
\end{Def}

\mainone*

\begin{proof}
  By Proposition~\ref{prop:computable_omega} and Theorem~\ref{Thm:2n-n}, every $1$-REA degree contains an NCR real. Part (a) now follows inductively using Theorem~\ref{Thm:2n-n}. Part (b) follows from Theorem~\ref{Thm:modulelevelinfty}.
\end{proof}

The result actually holds in a slightly stronger form in that both kind of
degrees contain NCR reals \emph{of level $\omega$}, that is, reals that are
non-$\mu$-random of level $\omega$ for every continuous measure $\mu$ (see
~\cite{li:thesis}). However, for our main applications the form stated here is
quite sufficient.

\medskip
Since every $\Delta^0_2$ degree has a self-modulus, we obtain

\maintwo*

Furthermore, if a real $B$ has a self-modulus, by using the relativized version of Shoenfield's Limit Lemma, we can prove the above result also holds for any $\Delta^0_2(B)$ real above $B$, so we have the following.

\begin{Corollary} \label{cor:delta02-NCR}
  If a real $B$ has a self-modulus, then every $\Delta^0_2(B)$ real above $B$ contains an NCR element.
\end{Corollary}

We can also apply our techniques to prove the existence of weakly generic reals in NCR.
\begin{Thm}\label{thm:ncr_1generic}
  For every self-modulus degree above $0'$, there exists a weakly $1$-generic NCR real in it.
\end{Thm}

\begin{proof}
  Assume $A=a_0 \, a_1 \, a_2 \, a_3 \dots$ and $f_A\equiv_T A$ is a self-modulus of $A$. Without loss of generality, we can assume $f_A(n)$ is increasing. Let $W_n$ be $n$-th $\Sigma^0_1$ set of binary strings.

  We define our first string $B_0$ as
  $$B_0 =1^{f_A(0)}\Conc 0\Conc a_0,$$

  And define $\sigma_i$, $B_i$ inductively as
  \begin{equation*}
    \sigma_i := \begin{cases}
      \text{the smallest such } \tau & \text{if $\exists \tau \in W_i(B_i\Conc 1 \sqsubset \tau)$;} \\
      B_i\Conc 0                     & \text{otherwise.}
    \end{cases}
  \end{equation*}
  $$B_{i+1}:=\sigma_i\Conc 1^{f_A(\mid\sigma_i\mid)}\Conc 0\Conc a_{i+1}.$$

  Finally define $B$ as
  $$B:=\lim_{i\rightarrow \infty} B_i.$$

  Since $A>_T 0'$ A compute all $\sigma_i$, thus compute $B$. And $B$ can effectively recover all $B_i$, So B also compute $A$, thus $A\equiv_T B$.

  Moreover, the proof of Theorem \ref{Thm:modulelevelinfty} also can be applied to the $B$ we constructed here, so $B$ is NCR.

  Lastly we show $B$ is weakly 1-generic. If $W_i$ is a dense $\Sigma_0^1$ set, then $\sigma_i\in W_i$ and $\sigma_i$ is an initial segment of $B$, so $B$ is weakly 1-generic.
\end{proof}

Using similar ideas, one can construct 1-generic NCR reals. It is also possible, albeit more complicated, to construct an NCR real of minimal Turing degree. These constructions are given in~\cite{li:thesis}.

\section{Further applications and open questions}
\label{sec:6}

We can apply the techniques introduced in this paper to address a question
asked by Adam Day and Andrew Marks (private communication).

\begin{Def}
  Two reals $X_1,X_2 \in \Cant$ are \emph{simultaneously continuously random}
  if there exists a real $Z$ and a  measure $\mu$ such that $Z$ computes
  $\mu$ and both $X_1$ and $X_2$ are $\mu$-random relative to $Z$. If such $Z$ and
  $\mu$ do not exist,  $X_1,X_2$ are called \emph{never simultaneously
    continuously random} (NSCR).
\end{Def}

Day and Marks conjectured that $X_1$ and $X_2$ are NSCR if and only if at least one
of them is in NCR. We refute this conjecture by
constructing two reals $X_1$ and $X_2$ such that they are both random with respect to some continuous measure,
but for every measure $\mu$ for which $X_2$ is random, any representation of $\mu$ computes $X_1$.

Let $f(n)$ be a self-modulus of $0'$ and $X_1$ be a $\lambda$-random $\Delta^0_2$ real, where $\lambda$ is Lebesgue measure. It suffices to
find a real $X_2$ which random for some continuous $\mu$ and every  representation of a continuous measure
$\nu$ for which $X_2$ is random can compute a function which dominates $f(n)$.

We define
$$S_0:=\{1^{f(0)}\Conc 0 \Conc x \colon x \in \{0,1\} \}.$$
And
$$S_{n+1}:=\{\sigma \Conc 1^{f(\mid\sigma\mid)}\Conc 0 \Conc x \colon  \sigma \in S_n, x \in \{0,1\} \}.$$
Finally define
$$S:=\{Y \in \Cant \colon \forall n \exists \sigma_n\in S_n(\sigma_n\sqsubset Y)\}.$$
Suppose $\mu$ is a continuous measure with a representation $R_\mu$ that does
not compute any function dominating $f$. An argument similar to the proof of
Theorem~\ref{Thm:modulelevelinfty} yields that the set $T_k$ defined there is a
level-$k$ test. Moreover, by the definition of $S$, every real in $S$ is
covered by $T_k$. Therefore, any element in $S$ can only be random for a
measure all of whose representations compute a function dominating $f$. It
follows that any element of $S$ is NSCR with $X_1$.

% Then for any continuous measure $\mu$, if there is a representation $R_\mu$ cannot compute any function which dominates $f(n)$, by the same proof of \ref{Thm:modulelevelinfty}, the set $T_n$ we defined in the proof are level-n $R_\mu$-$\mu$-tests which cover every element in $S$. Thus for any continuous measure $\mu$, if there is a representation $R_\mu$ cannot compute any function which dominates $f(n)$, every element in $S$ would be non-random in $\mu$. In the other words, element in $S$ can only be random in measure whose representations are all compute some functions dominates $f(n)$, thus compute $0'$ and $X_1$. So every element in $S$ are NSCR with $X_1$.   

It remains to show that there is a element in $S$ which is random with respect
to a continuous measure. This easily follows from the fact that NCR is
countable (see~\cite{reimann2015measures}), but we can give a direct argument
as follows:
It follows from the construction of $S$ that $S$ is a perfect subset of
$\Cant$. By distributing a unit mass uniformly along $S$, we obtain a
continuous measure whose support is $S$ and we can choose any real that is
random with respect to this measure and obtain

% We define a bijection $F(X):2^\omega \rightarrow S$ by putting
% $$F_0(X):=1^{f(0)}\frown 0 \frown x_0$$
% $$F_{n+1}(X):=F_n(X)\frown 1^{f(\midF_n(X)\mid)}\frown 0 \frown x_{n+1}$$
% $$F(X):=\lim_n F_n(X)$$
% So there must by an element $X_2$ in $S$ which is not NCR. And we also can find such an $X_2$ by construct a measure $\mu_f$ from above bijection as $\mu_f([\sigma])=\lambda(F^{-1}([\sigma]))$, in which $\mu_f(S)=1$, so there must be some element $X_2$ in $S$ which is random with respect to $\mu_f$. 
\begin{Corollary}
  There are non-NCR reals $X_1$ and $X_2$ which are NSCR.
\end{Corollary}

%
% NEW part (02 11 22)
%
Theorem~\ref{Thm:modulelevelinfty} can in fact be used to construct a whole sequence of mutually NSCR reals. This answers a question posed by Yu Liang.

\begin{Thm}
  There exists a countable sequence of mutually NSCR reals.
\end{Thm}

\begin{proof}
  For any positive natural number $n$, let $f_n$ be a self-modulus function of $0^{(n)}$. We also define a measure $\mu_n$ on $\Cant$ by requiring
  \[
    \mu_n\Cyl{a_0\Conc 0^{f_n(0)} \Conc a_1\Conc 0^{f_n(1)} \Conc \ldots \Conc a_i}=2^{-i-1},
  \]
  where the $a_0,a_1,a_2,...a_i$ are arbitrary bits in $\{0,1\}$.
  Since $0^{(n)}$ computes (a representation of) $\mu_n$, there exists a $0^{(n+1)}$-computable real $X_n$ random for $\mu_n$.
  Moreover, if $X_n$ is random in $\mu_n$, it must be of the form
  \begin{equation} \label{equ:X_n-prefix} \tag{*}
    a_0\Conc 0^{f_n(0)}\Conc a_1\Conc 0^{f_n(1)} \Conc \ldots \Conc a_i\Conc 0^{f_n(i)}\Conc \ldots
  \end{equation}
  for a sequence of $a_i$ in  $\{0,1\}$, otherwise it would be contained in a $\mu_n$-null cylinder.

  We claim the $\{X_n\}_{n\in \N}$ are mutually NSCR. We show this by contradiction. Assume there are natural numbers $m< n\in \N$, a real real $Z$ and a measure $\mu$ with a $Z$-computable representation, such that $X_m$ and $X_n$ are both $\mu$-random relative to $Z$. Since $X_n$
  is of the form \eqref{equ:X_n-prefix},
  by the same argument as in the proof of Theorem~\ref{Thm:modulelevelinfty},
  $Z$ must compute a function that dominates $f_n$, thus $Z$ computes $0^{(n)}$.

  Since $X_m$ is $0^{(m+1)}$-computable and $m<n$, it follows that $X_m$ is $Z$-computable, and hence cannot be $\mu$-random relative to $Z$, contradiction.
\end{proof}

\section{Questions and conclusion}
The exact distribution of NCR reals in $\Delta^1_1$ remains unknown. Taking
into account the results of this paper, the following questions seem
particularly interesting.

%First of all, there are some continuous measure $\mu$ in which non-$\mu$ random of level 1 is strictly stronger than non-$\mu$ Martin-L\"{o}f random. For example, let $\lambda$ be the Lebesgue measure, define a measure $\mu$ as
%\begin{equation*}
%\mu(\Cyl{\tau}) := \begin{cases}
%1 &\text{if $\mid\tau \mid = 0$;}\\
%0 &\text{if $``0" \sqsubset \tau$;}\\
%2^{-\mid\tau\mid+1} &\text{if $``1" \sqsubset \tau$.}
%\end{cases}
%\end{equation*}
%Every real extend $``0"$ is non-$\mu$ Martin-L\"{o}f random since $\Cyl{0}$ has measure 0. Meanwhile, $\mu$-random reals of level 1 are the same as Martin-L\"{o}f random $w.r.t.$ Lebesgue measure because the dissipation function $h_\mu (l)=l-1$ for all $
%l>0$.   

%However, whether NCR of level 1 is strictly stronger than $NCR$ in
%the sense of Martin-L\"{o}f is still unknown. In the other word, we may ask:
%\begin{quote}
%\em Are there any NCR real which is level-1 random w.r.t. some continuous measure?
%\end{quote}

Following the results of Section~5, we can ask how strong the
relation between $\Delta^1_1$ degrees containing NCR reals degrees with a
self-modulus is. In particular, does the following hold:
\begin{quote}
  \em If $\mathcal{D}$ contains an NCR real, must $\mathcal{D}$ have a
  self-modulus?
\end{quote}
If the answer to this question is negative, then we can ask a weaker one:
\begin{quote}
  \em If $\mathcal{D}$ contains a real that is NCR of level $\omega$, must
  $\mathcal{D}$ have a self-modulus?
\end{quote}
On the other hand, our results only concern the \emph{existence} of \emph{some}
NCR elements in Turing degrees, while \cite{barmpalias2012k} shows that
\emph{all} reals in an incomplete r.e.\ degree are NCR. Thus, we may also ask:
\begin{quote}
  \em Are there any other Turing degree not below any incomplete r.e.\ degree in
  which every real is in NCR?
\end{quote}

\bigskip
As NCR is $\Pi^1_1$ set of reals, it has a $\Pi^1_1$ rank function (see for
example~\cite{Kechris:1995a}). It is an open problem to find a ``natural'' rank
function for NCR which reflects the stratified complexities of elements in NCR
in a more informative way. Such a rank function is arguably needed to shed more
light on the structure of NCR in the Turing degrees.
Theorem~\ref{thm:ncr_1generic} immediately implies that a rank based on the
Cantor-Bendixson derivative will not work -- NCR is a proper superset of the
members of countable $\Pi^0_1$ classes. (This follows also from the
Barmpalias-Greenberg-Montalbán-Slaman result~\cite{barmpalias2012k}, of
course.)

Restricted to $\Delta^0_2$, the picture is a little clearer. We now know that
every
$\Delta^0_2$ Turing degree contains an NCR real (Corollary~\ref{cor:delta02-NCR}),
and every degree below an incomplete r.e.\ degree is completely
NCR~\cite{barmpalias2012k}. Moreover, using the connection between the
granularity function and the settling function, it is possible to show that
$\operatorname{NCR}\cap \Delta^0_2$ is an arithmetic set of
reals~\cite{reimann-slaman:unpub}\footnote{A proof of this result can be found in~\cite{li:thesis}.}. Unfortunately, few of the techniques developed so
far (including the ones developed in this paper) seem to extend easily higher
up the arithmetic hierarchy. The question whether, for example,
$\operatorname{NCR}\cap \Delta^0_2$ is arithmetic remains open.
\bibliography{sn-bibliography}% common bib file


\end{document}